\documentclass{amsart}
\usepackage{graphicx}
\usepackage{latexsym}
\usepackage{hyperref}
\usepackage{amsfonts}
\usepackage[all]{xy}
\usepackage{amssymb, mathrsfs, amsfonts, amsmath}
\usepackage{amsbsy}
\usepackage{amsfonts}
\usepackage{setspace}
\usepackage{cancel}
\usepackage{xcolor}

\usepackage{mathrsfs}
\usepackage{amsbsy}

\newtheorem{corollary}{Corollary}

\newtheorem{lemma}{Lemma}
\newtheorem{proposition}{Proposition}
\newtheorem{remark}{Remark}
\newtheorem{theorem}{Theorem}

\newtheorem{example}{Example}

\numberwithin{equation}{section}

\global\long\def\Vol{\mathrm{Vol}}

\makeatletter
\@namedef{subjclassname@2020}{%
  \textup{2020} Mathematics Subject Classification}
\makeatother

\begin{document}
\title[ $\rho$-Einstein solitons]{Geometric and analytical results\\ for $\rho$-Einstein solitons}

\author{Caio Coimbra}

\address[C. Coimbra]{Departamento  de Matem\'atica, Universidade Federal do Cear\'a - UFC, Campus do Pici, Av. Humberto Monte, Bloco 914,
60455-760, Fortaleza - CE, Brazil}
\email{caioadler@alu.ufc.br}

\thanks{C. Coimbra was partially supported by CAPES/Brazil - Finance Code 001}

	\thanks{Corresponding Author: C. Coimbra  (caioadler@alu.ufc.br)}

\keywords{$\rho$-Einstein solitons; spectrum gap; volume growth estimate}

\subjclass[2020]{Primary 53C20, 53C25; Secondary 53C65.}

\date{December 19, 2024}

\begin{abstract}
In this article, we study geometric and analytical features of complete noncompact $\rho$-Einstein solitons, which are self-similar solutions of the Ricci-Bourguignon flow. 
We study the spectrum of the drifted Laplacian operator for complete gradient shrinking $\rho$-Einstein solitons. Moreover, similar to classical results due to Calabi-Yau and Bishop for complete Riemannian manifolds with nonnegative Ricci curvature, we prove new volume growth estimates for geodesic balls of complete noncompact $\rho$-Einstein solitons. In particular, the rigidity case is discussed. In addition, we establish weighted volume growth estimates for geodesic balls of such manifolds. 
\end{abstract}

\maketitle
\section{Introduction}\label{intro}

The classical Lichnerowicz theorem \cite{lichnerowicz1958geometrie} states that if $(M^n,\,g)$ is a compact (without boundary) Riemannian manifold with bounded Ricci curvature $Ric \geq \alpha,$ where $\alpha$ is a positive constant, then the first nonzero eigenvalue $\lambda_1(\Delta)$ of the Laplacian operator $\Delta,$ also known as spectrum gap, must satisfy $\lambda_1(\Delta)\geq \frac{n}{n-1}\alpha.$ Furthermore, Obata's theorem \cite{obata1962certain} says that the equality holds if and only if $(M^n,\,g)$ is an $n$-dimensional sphere with constant sectional curvature $\frac{\alpha}{n-1}$. This raised the question whether a similar result holds true for smooth metric measure spaces. In this context, we recall that a smooth metric measure space $(M^n,\,g,\,e^{-f}dV)$ is a complete $n$-dimensional Riemannian manifold with a potential function $f:M\to \Bbb{R}$ and the weighted volume $e^{-f}dV$ in $M.$ For such spaces, it is more natural to consider the drifted Laplacian operator

\begin{equation*}
    \Delta_f = \Delta - \langle\nabla f,\nabla\,\cdot\,\rangle.
\end{equation*} This comes from the fact that $\Delta_f$ is a densely defined self-adjoint operator in $L^2(M,\,e^{-f}dv)$ and hence, for $u,\,v \in C_{0}^{\infty}(M),$ one sees that
\begin{equation*}
    \int_M u\Delta_fve^{-f}\,dV = -\int_M\langle\nabla u,\nabla v\rangle e^{-f}\,dV.
\end{equation*} Moreover, instead of the usual Ricci tensor, we may consider the Bakry-\'Emery Ricci tensor given by
\begin{equation*}
    Ric_f := Ric + \nabla^2 f,
\end{equation*} where $\nabla^2 f$ stands for the hessian of the potential function $f.$ However, it is important to highlight that, in this case,
certain topological differences arise. For instance, the Bonnet-Myers theorem does not hold under the assumption that $Ric_f \geq \delta > 0$ when $f$ is non-constant, and the manifold may fail to be compact, as exemplified by the Gaussian shrinking soliton $(\mathbb{R}^n,\,g_{_{can}},\,f(x)=\frac{|x|^2}{4}).$

By the works of Bakry-\'Emery \cite{bakry2006diffusions}, Morgan \cite{morgan2005manifolds}, Hein-Naber \cite{hein2014new} and Cheng-Zhou \cite{cheng2017eigenvalues}, it is known the following Lichnerowicz-Obata type theorem for smooth metric measure spaces.
 
\begin{theorem}[\cite{bakry2006diffusions,cheng2017eigenvalues,hein2014new,morgan2005manifolds}]
\label{theo1}
    Let $(M^n,\,g,\,e^{-f}dV)$ be a complete smooth metric measure space with $Ric_f \geq \frac{\alpha}{2}g$ for some positive constant $\alpha.$ Then the spectrum of the drifted Laplacian operator $\Delta_f$ is discrete and the first nonzero eigenvalue, denoted by $\lambda_1(\Delta_f)$, must satisfy
    \begin{equation}\label{eq1}
        \lambda_1(\Delta_f) \geq \frac{\alpha}{2}.
    \end{equation} Moreover, equality holds in (\ref{eq1}) with multiplicity $k\geq 1$ if and only if 
  \begin{enumerate}
      \item $1\leq k\leq n$;
      \item $M$ is a noncompact manifold which is isometric to $\Sigma^{n-k}\times \mathbb{R}^k$ with the product metric for some complete $(n-k)$-dimensional manifold $(\Sigma,g_{\Sigma})$ satisfying $Ric^{\Sigma}_f\geq \frac{\alpha}{2}g_{\Sigma}$ and $\lambda_1(\Delta^{\Sigma}_f) > \frac{\alpha}{2}$;
      \item By passing an isometry, for $(x,t) \in \Sigma^{n-k}\times\mathbb{R}^k,$
      \begin{equation*}
          f(x,t) = f(x,0) + \frac{\alpha}{2}|t|^2.
      \end{equation*}
  \end{enumerate}
\end{theorem}

In this article, we investigate the geometric properties of gradient $\rho$-Einstein solitons. In particular, we establish a result analogous to Theorem \ref{theo1} for this class of manifolds. Before stating our main results, let us recall the definition of such manifolds. For a given $\rho \in \mathbb{R},$ we say that $(M^n,\,g,\,f,\lambda)$ is a gradient $\rho$-Einstein soliton if it satisfies the equation
\begin{equation}\label{eqfund}
    Ric + \nabla^2 f = (\rho R + \lambda)g.
\end{equation} Following the terminology used for Ricci solitons, we say that the  $\rho$-Einstein soliton is shrinking, steady or expanding if $\lambda > 0,$ $\lambda = 0$ or $\lambda <0,$ respectively. When $\rho = \frac{1}{2(n-1)},$ it is called {\it Schouten soliton}; see \cite{catino2016gradient}. Moreover, notice that the gradient Ricci soliton equation is obtained when $\rho = 0$ in (\ref{eqfund}).

Gradient $\rho$-Einstein solitons were first introduced in \cite{catino2016gradient}. According to the work of Catino, Mazzieri and Mongodi \cite{catino2015rigidity}, these solitons arise as self-similar solutions to the Ricci-Bourguignon flow, originally introduced by Bourguignon in \cite{bourguignon1981ricci},
\begin{equation*}
    \frac{\partial}{\partial t}g_t = -2(Ric_{g_t} - \rho Rg_t);
\end{equation*} see also \cite{catino2017ricci}.  In the last years, several topological, geometric and analytical features concerning Schouten solitons have also been proven. For instance, Catino and Mazzieri \cite{catino2016gradient} showed that a complete steady Schouten soliton must be Ricci flat. Moreover, they proved that  $3$-dimensional Schouten solitons are isometric to $\mathbb{R}^3,\mathbb{S}^3$ or $\mathbb{S}^2\times\mathbb{R}.$ While Borges \cite{borges2024rigidity} classified gradient Schouten solitons with vanishing Bach tensor. Furthermore, he \cite{borges2022complete} obtained interesting results regarding the asymptotic behavior of the potential function $f$ and the norm of its gradient. Cunha, Lemos and Roing \cite{cunha2023ricci} established conditions under which $\rho$-Einstein solitons have constant scalar curvature; see also  \cite{shaikh2021some}. Despite these results, it remains of interest to show that such manifolds are analytic. In \cite{catino2015rigidity}, Catino et al. proved that if $\rho \notin \{\frac{1}{n},\frac{1}{2(n-1)}\},$ then the metric $g$ and the potential function $f$ must be analytic.

Among the examples of $\rho$-Einstein solitons, we may mention that Einstein manifolds are natural examples of $\rho$-Einstein solitons with constant potential function $f.$ However, it is interesting to obtain examples with non-constant potential function $f,$ i.e., a {\it nontrivial} $\rho$-Einstein soliton. A classical nontrivial example can be obtained in the generalized cylinder (see \cite{borges2022complete}). To be precise, for $n \geq 3,$ $k \leq n,$ $\lambda \in \mathbb{R}$ and $\rho \in \mathbb{R}$ with $\rho \neq \frac{1}{k},$ we consider a $k$-dimensional Einstein manifold $(\Sigma^k,\,g_{\Sigma})$ with scalar curvature $R_{\Sigma} = \frac{k\lambda}{1-\rho k}.$ Besides, for $(x,p) \in \mathbb{R}^{n-k}\times \Sigma^k$, it suffices to take the potential function
\begin{equation*}
    f(x,p) = \frac{1}{2}\Bigg(\frac{\lambda}{1-\rho k}\Bigg)|x|^2
\end{equation*} to conclude that  $(\mathbb{R}^{n-k}\times \Sigma^k,g,f,\lambda)$ is a nontrivial $\rho$-Einstein soliton, where $g$ is the product metric. At the same time, we highlight that $\rho$-Einstein solitons with constant scalar curvature are precisely gradient Ricci solitons; see \cite{Hamilton}. Hence, it is very important to present examples of $\rho$-Einstein solitons with non-constant scalar curvature. As it was observed by Gomes and Agila, $\mathbb{H} \times_{h}\mathbb{F}^2$ and $\mathbb{R} \times_{h}\mathbb{F}^2,$ where $\mathbb{F}^2$ is a complete $2$-dimensional Ricci flat manifold, are nontrivial $\rho$-Einstein solitons with non-constant scalar curvature; for more details, see Examples \ref{ex1} and \ref{ex2} in Section \ref{prelim}; see also \cite{agila2024geometric}.

We are now ready to state our first main result, which establishes a Lichnerowicz-Obata-type theorem for gradient shrinking $\rho$-Einstein solitons. More precisely, we have the following result.

\begin{theorem}
\label{theo2}
    Let $(M^n,\,g,\,f,\,\lambda)$ be a complete gradient shrinking $\rho$-Einstein soliton with $\rho>0$ and nonnegative scalar curvature. Then the following assertions hold:
    \begin{enumerate}
        \item The spectrum of $\Delta_f$ is discrete;
        \item $\lambda_1(\Delta_f) \geq \lambda$;
        \item Equality holds in assertion $(2)$ if and only if $(M^n,g,f)$ is $(\mathbb{R}^n,\delta_{ij},\frac{\lambda}{2}|x|^2)$, i.e., the Gaussian shrinking soliton $(\mathbb{R}^n,\delta_{ij},\frac{|x|^2}{4}),$ up to scaling.
    \end{enumerate}
\end{theorem}

It is worth noting that the discreteness and lower bound of the spectral gap also follow from Theorem \ref{theo1}. However, a key advantage of our approach
lies in the rearrangement of its proof, which allows us to establish the rigidity case.

As a consequence of Theorem \ref{theo2} combined with the scalar curvature estimates proved by Catino and Mazzieri \cite[Corollary 5.2]{catino2016gradient} and Borges \cite[Theorem 1.1]{borges2022complete}, we get the following corollary for Schouten solitons.

\begin{corollary}
    Let $(M^n,\,g,\,f,\,\lambda)$ be a complete gradient shrinking Schouten soliton. Then the following assertions hold:
        \begin{enumerate}
        \item The spectrum of $\Delta_f$ is discrete;
        \item $\lambda_1(\Delta_f) \geq \lambda$;
        \item Equality holds in assertion $(2)$ if and only if $(M^n,\,g,\,f)$ is $(\mathbb{R}^n,\delta_{ij},\frac{\lambda}{2}|x|^2)$, i.e., the Gaussian shrinking soliton $(\mathbb{R}^n,\delta_{ij},\frac{|x|^2}{4}),$ up to scaling.
    \end{enumerate}
\end{corollary}

A precise understanding of the volume growth rate is a key geometric property from which various other features of the underlying Riemannian
manifold can be derived. A classical theorem due
to Calabi \cite{calabi} and Yau \cite{yau1} asserts that the geodesic balls of complete noncompact manifolds with nonnegative Ricci curvature have at least linear volume growth. While the classical Bishop volume comparison theorem \cite{bishop} guarantees that the geodesic balls of complete noncompact Riemannian manifolds with nonnegative Ricci curvature have at most polynomial volume growth. 
Similar results were obtained for gradient Ricci solitons in \cite{cao2010complete,chan2022volume,munteanu2009volume,Natasa} and quasi-Einstein manifolds in, e.g., \cite{barros2015bounds,BRR,cheng2022volume}. Recently, Borges \cite{borges2022complete} adapted some ideas by Cao and Zhou \cite{cao2010complete} in order to prove volume growth estimates for Schouten solitons. To be precise, he showed that given a complete noncompact shrinking Schouten soliton, there are positive constants $C_1$, $C_2$ and $r_0$, depending only on $\lambda$ and $n,$ such that
    \begin{equation*}
        C_1r^{\frac{n}{2}-\frac{(n-2)\theta}{4(n-1)\lambda}}\leq Vol(B_q(r))\leq C_2r^{n-\frac{(n-2)\delta}{2(n-1)\lambda}},
    \end{equation*}
   for any $r>r_0$, where $q \in M$, $\delta = \inf_{p\in M}R(p)$ and $\theta = \sup_{p\in M}R(p).$  The proof of this result relies on the behavior of the potential function $f.$ We highlight that it is not known whether the same approach is valid for $\rho$-Einstein solitons in general. In the same context, Munteanu and Wang \cite[Theorem 1.4]{munteanu2014geometry} showed that any $n$-dimensional smooth metric measure space $(M^n,\,g,\,e^{-f}dv)$ with $Ric_f \geq \frac{1}{2}$ and $|\nabla f|^2 \leq f$ must satisfy  $Vol(B_p(r)) \leq c(n)r^n$, for all $r>0$, where $c(n)$ is a constant depending only on the dimension of the manifold. While Cheng, Ribeiro and Zhou \cite[Theorem 5]{cheng2022volume} obtained the precise value of the constant $c(n)$ for gradient shrinking Ricci solitons. Similar estimates are interesting for  $\rho$-Einstein solitons. Here, by adapting some techniques outlined in \cite{cheng2022volume}, we establish the following volume growth estimate for geodesic balls of complete noncompact $\rho$-Einstein solitons.

\begin{theorem}\label{theo4}
    Let $(M^n\,,g,\, f)$ be a complete noncompact $n$-dimensional gradient shrinking $\rho$-Einstein soliton with $\rho > 0$ and $R \geq 0.$ Then, for all $r>0,$ the volume of the geodesic ball $B_p(r)$ must satisfy
    \begin{equation}\label{eq11}
        Vol(B_p(r)) \le  \int_{\mathbb{S}^{n-1}} \int_0^{r}e^{\Phi}r^{n-1}\,drd\theta,
    \end{equation}
    where  $$\Phi= -\frac{\lambda r^2}{6}+f(\theta, r)+f(p)-\frac{2}{r}\int_{0}^{r}f(\theta, s) \,ds.$$ Moreover, equality holds in (\ref{eq11}) for all $r>0$ if and only if $(M^n,\,g,\,f)$ is $(\mathbb{R}^n,\delta_{ij},\frac{\lambda|x-p|^2}{2})$, i.e., the Gaussian shrinking soliton $(\mathbb{R}^n,\delta_{ij},\frac{|x|^2}{4}),$ up to scaling and translation.
\end{theorem}

\begin{remark}
We point out that Theorem \ref{theo2} and Theorem \ref{theo4} are also true if we assume $\rho < 0$ and $R \leq 0,$ with a minor difference in the rigidity part of Theorem \ref{theo2}. In particular, the proof of this fact is essentially the same. 
\end{remark}

In general, as in the case of smooth metric measure spaces, it is also very important to obtain weighted volume growth estimates for geodesic balls for $\rho$-Einstein solitons. This is due to the fact that the existence of the potential function $f$ provides us useful analytical information associated with the drifted Laplacian operator. In \cite[Lemma 1]{agila2024geometric}, Agila and Gomes proved an estimate for the weighted volume of geodesic balls under suitable conditions involving the scalar curvature and the potential function. However, it is not hard to check that the integral expression obtained by them in the result diverges. Therefore, it is interesting to obtain new (refined) volume growth estimates for such manifolds. As a consequence of the proof of Theorem \ref{theo4}, we have the following result.

\begin{theorem}\label{theo5}
    Let $(M^n,\,g,\, f)$ be a complete noncompact $n$-dimensional gradient $\rho$-Einstein soliton. Then
    \begin{equation}\label{eq120}
        Vol_{f}(B_p(r)) \le  \int_{\mathbb{S}^{n-1}} \int_0^{r}e^{\Psi}r^{n-1}\,drd\theta,
    \end{equation}
    where  $$\Psi= -\frac{\lambda r^2}{6}+ f(p)-\frac{2}{r}\int_{0}^{r}f(\theta, s)ds -\frac{1}{r}\int_{0}^r\int_{0}^s t\rho R\,dtds.$$ 
\end{theorem}

Combining Theorem \ref{theo5} and the potential function estimates obtained by Munteanu and Wang \cite[Theorem 0.4]{munteanu2015topology}, we get the following corollary.

\begin{corollary}\label{cor2}
         Let $(M^n,g, f)$ be a complete noncompact $n$-dimensional gradient shrinking $\rho$-Einstein soliton with $\rho R  \geq \delta > -\lambda.$ Then there exists $r_0 > 0$ such that, for all $r \geq r_0,$ 
     \begin{equation}
         Vol_{f}(B_p(r)) \leq \int_{\mathbb{S}^{n-1}}\int_{0}^{r}e^{\frac{r^2}{2}}r^{n-1}\,dr d\theta.
     \end{equation}
\end{corollary}

     In other words, Corollary \ref{cor2} asserts that the weighted volume of a complete noncompact gradient shrinking $\rho$-Einstein soliton with $\rho R  \geq \delta > -\lambda$ is less than or equal to the weighted volume of the Gaussian shrinking soliton.

\begin{remark}
We highlight that under the hypothesis of Corollary \ref{cor2}, it follows that
 $Ric+\nabla^2 f \geq (\delta+\lambda)g,$ which is a positive constant. In this situation, Wei and Wylie \cite[Theorem 3.1]{wei2009comparison} proved a similar estimate.
\end{remark}

The article is organized as follows. In Section \ref{prelim} we present some basic results on $\rho$-Einstein solitons and explicit examples. Section \ref{proofM} contains the proofs of Theorems \ref{theo2}, \ref{theo4} and \ref{theo5}, respectively.

\section{Preliminaries}
\label{prelim}

In this section, we present basic facts that are useful for the establishment of the main results. Moreover, we will describe some examples of gradient $\rho$-Einstein solitons. To begin with, for a fixed $p \in M,$ in terms of the polar normal coordinates at $p,$ we may write the volume element as $J(\theta,r) dr \wedge d\theta$, where $d\theta$ is the volume element of the $(n-1)$-dimensional sphere $\mathbb{S}^{n-1}.$ From the Gauss lemma, it is known that the area element of the geodesic sphere is given by $J(\theta,r)d\theta.$ Now, consider $x=(\theta,r)$, a point outside of the cut-locus $\mathcal{C}(p)$ of $p,$ and define
    \begin{equation*}
w(\theta,r)=\frac{\partial}{\partial r}\log J(\theta,r)=\frac{\frac{\partial J}{\partial r}}{J}(\theta,r).
\end{equation*} The second area variational formula in polar coordinates is given by
\begin{equation}\label{eqvar2}
\frac{\partial^2J}{\partial r^2}(\theta,r) = -\sum_{i,j=1}^{n-1}h_{ij}^2(\theta,r)J(\theta,r) - Ric\Big(\frac{\partial}{\partial r},\frac{\partial}{\partial r}\Big)J(\theta,r) + \frac{\Big(\frac{\partial J}{\partial r}\Big)^2}{J}(\theta,r),  
\end{equation}
where $h_{ij}(\theta,r)$ stands for the second fundamental form of $\partial B_p(r)$. Moreover, notice that
\begin{equation}\label{eqvar3}
    \frac{\partial w}{\partial r}(\theta,r) = \frac{\frac{\partial^2J}{\partial r^2}}{J}(\theta,r) - w^2(\theta,r).
\end{equation}
Plugging \eqref{eqvar3} into \eqref{eqvar2}, one sees that
\begin{equation*}
    \frac{\partial w}{\partial r}(\theta,r) + w^2(\theta,r) = -\sum_{i,j=1}^{n-1}h_{ij}^2(\theta,r) - Ric\Big(\frac{\partial}{\partial r},\frac{\partial}{\partial r}\Big) + \frac{\Big(\frac{\partial J}{\partial r}\Big)^2}{J^2}(\theta,r).
\end{equation*}
Since $h_{ij}$ is a 2-tensor, we have $|h|^2 \geq \frac{H}{n-1,}$ where $H(\theta,r)$ is the mean curvature of $\partial B_p(r)$. Therefore, it follows that
\begin{equation}
    \frac{\partial w}{\partial r}(\theta,r)+ \frac{1}{n-1}H^2(\theta,r) + Ric\Big(\frac{\partial}{\partial r},\frac{\partial}{\partial r}\Big)  \leq 0.
\end{equation}
Finally, observe that the first area variational formula gives $H(\theta,r) = w(\theta,r)$. From this, one concludes that $w(\theta,r)$ must satisfy
\begin{equation}\label{eqvar1}
w'(\theta,r)+\frac{1}{n-1}w^{2}(\theta,r)+Ric\Big(\frac{\partial}{\partial r},\frac{\partial}{\partial r}\Big)\le0,
\end{equation} where $w' := \frac{\partial w}{\partial r}$ and $Ric\Big(\frac{\partial}{\partial r},\frac{\partial}{\partial r}\Big)$ stands for the Ricci curvature in the radial direction; for more details see, e.g., \cite{li2012geometric}.

Proceeding, we recall special features established by Catino et al. \cite{catino2015rigidity} for $\rho$-Einstein solitons. The case of gradient Ricci solitons was proved by Hamilton \cite{hamilton1993formations}.

\begin{lemma}[\cite{catino2015rigidity}]
\label{lem1}
Let $(M^n,\,g,\,f,\,\lambda)$ be a gradient $\rho$-Einstein soliton. Then the following equations hold:
\begin{enumerate}
\item $\Delta f = (n\rho -1)R + n\lambda;$
\item $(1-2(n-1)\rho)\nabla R = 2Ric(\nabla f);$
\item $(1-2(n-1)\rho)\Delta R = \langle\nabla R,\nabla f\rangle + 2(\rho R^2 - |Ric|^2 + \lambda R).$
\end{enumerate}
\end{lemma}

In \cite{borges2022complete}, Borges obtained some useful properties regarding the scalar curvature and the norm of the gradient of the potential function.

\begin{proposition}[\cite{borges2022complete}]
\label{propborges}
    Let $(M^n,\,g,\,f,\,\lambda)$ be a complete noncompact Schouten soliton. Suppose that $\lambda >0$ ($\lambda<0$, respectively). Then the potential function $f$ attains a global minimum (maximum, respectively) and it is unbounded from above (below, respectively). Furthermore, we have:
    \begin{equation*}
        0\leq R\lambda\leq 2(n-1)\lambda^2
    \end{equation*} and
    \begin{equation*}
        2\lambda(f-f_0)\leq |\nabla f|^2\leq 4\lambda(f-f_0),
    \end{equation*} where $f_0 = \min_{p\in M}f(p)$\,\,\,($f_0 = \max_{p\in M}f(p)$, respectively).
\end{proposition}

For the case  $\lambda > 0,$ Catino and Mazzieri \cite{catino2016gradient} established lower bounds for the scalar curvature using the Ricci-Bourguignon flow. To the best of our knowledge, there is no version of Proposition \ref{propborges} for $\rho$-Einstein solitons in general.

In \cite{borges2022complete}, Borges also obtained the following proposition concerning the asymptotic behavior of the potential function of a shrinking Schouten soliton.

\begin{proposition}[\cite{borges2022complete}]
    Let $(M^n,\,g,\,f,\,\lambda)$ be a complete noncompact shrinking Schouten soliton. Then we have:
    \begin{equation*}
        \frac{\lambda}{4}(d(p,q) - A_1)^2 \leq f(p)- f_0\leq \lambda(d(p,q) + A_2)^2,
    \end{equation*}
    where $f_0 = \min_{p\in M} f(p)$ and $q$ is a point in $M,$ $A_1$ and $A_2$ are positive constants depending only on $\lambda$ and the unit ball $B_q(1)$ with $d(p,q)>2.$ 
\end{proposition}

In the rest of this section, we are going to present some nontrivial examples of $\rho$-Einstein solitons; for more details, see, e.g., \cite{agila2024geometric}.

\begin{example}
\label{ex1}
Consider the standard $1$-dimensional hyperbolic space $\mathbb{H}$ and let  $\mathbb{F}^2$ be a complete $2$-dimensional Ricci flat manifold. So, we choose $h(x) = \coth(x)$ and $f(x) = \frac{2}{3}\log (\cosh (x)).$ Besides, $\mathbb{H} \times_{h}\mathbb{F}^2$ with the warping metric and potential function $f$ satisfies
   \begin{equation*}
       Ric+\nabla^2f = -\frac{2 + 4\cosh(2x)}{3\cosh^4(x)}.
   \end{equation*} Therefore, it defines a steady $\rho$-Einstein soliton with $\rho = \frac{1}{3}$ and nonconstant scalar curvature given by
   \begin{equation*}
       R = -\frac{2+4\cosh(2x)}{\cosh^4(x)}.
   \end{equation*}
\end{example}

Reasoning as in the previous case, it is not hard to check the following example.

\begin{example}
\label{ex2}
    Consider the metric $\cosh^2(x)\delta$ in $\mathbb{R}$ and let $\mathbb{F}^2$ be a complete $2$-dimensional Ricci flat manifold.  By taking $h=\cosh(x) $ and $f = \frac{1}{12}(8\log(\cosh (x) + \cos (2x)),$ one sees that $\mathbb{R} \times_{h}\mathbb{F}^2$ satisfies
        \begin{equation*}
        Ric+\nabla^2f= -\frac{3-\sinh^4(x)}{3\cosh^4(x)},
    \end{equation*} with $h$ as warping function and $f$ as potential function. So, $\mathbb{R} \times_{h}\mathbb{F}^2$ is a gradient shrinking $\rho$-Einstein soliton with $\rho = \lambda = \frac{1}{3}$ and scalar curvature given by
    \begin{equation*}
        R = -\frac{3+\cosh(2x)}{3\cosh^4(x)}.
    \end{equation*}
    \end{example}

    As it was pointed out by Agila and Gomes \cite{agila2024geometric}, Examples \ref{ex1} and \ref{ex2} have complete metrics. In the sequel, we present an example endowed with non-complete metric.

\begin{example}
\label{ex3}
    Consider $\mathbb{R}^n,$ $n\geq 3$ with coordinates $x=(x_1,\,\ldots,\,x_n)$ and metric $g=e^{2\xi}\delta,$ where $\xi = \sum_{i=1}^{n}\alpha_i x_i$ and $\sum_{i=1}^{n}\alpha_i^{2}= 1.$ Again, let $\mathbb{F}^m$ be a  Ricci flat manifold. Besides, choosing $h = e^{\xi}$ and $f = \frac{c}{2}e^{\xi} - \frac{(2-m-n)}{2}\xi$ as warping function and potential function, respectively. So, it is not difficult to check that $\mathbb{R}^n \times_{h}\mathbb{F}^m$ must satisfy
        \begin{equation*}
        Ric+\nabla^2f = c + \frac{2-m-n}{2}e^{-2\xi}.
    \end{equation*} Then, $\mathbb{R}^n \times_{h}\mathbb{F}^m$ is a gradient Schouten soliton with $\lambda = c.$ Furthermore, the scalar curvature is given by
    \begin{equation*}
        R = -(m+n-2)(m+n-1)e^{-2\xi}.
    \end{equation*}
    \end{example}

Example \ref{ex3} guarantees that the completeness hypothesis cannot be removed in the results by Catino and Mazzieri \cite[Theorem 1.5]{catino2016gradient} and Borges \cite[Theorem 1.1]{borges2022complete}.

\section{Proof of the Main Results}
\label{proofM}
In this section, we shall present the proof of Theorems \ref{theo2}, \ref{theo4} and \ref{theo5} and Corollary \ref{cor2}

\subsection{Proof of Theorem \ref{theo2}}
\label{sec3}

\begin{proof} To begin with, we consider an eigenfunction of the drifted Laplacian operator $\Delta_f$, i.e., 
    $$
    \Delta_f u + \delta u = 0, \qquad \int_M u^2e^{-f}dV< \infty,
    $$
    where $u\in H^{1}(M,\mu)\cap C^{\infty}(M).$ From the weighted B\"ochner formula
    
     $$\frac{1}{2}\Delta_f |\nabla u|^2 = |\nabla^2 u|^2 + \langle \nabla u,\nabla (\Delta_f u)\rangle + Ric_{f}(\nabla u,\nabla u),$$ one sees that
   
\begin{eqnarray}\label{eq2}
                \frac{1}{2}\Delta_f |\nabla u|^2 &=& |\nabla^2 u|^2 -\delta|\nabla u|^2 + (\rho R + \lambda)|\nabla u|^2\nonumber\\
            &=& |\nabla^2 u|^2 + (\lambda-\delta)|\nabla u|^2 + \rho R|\nabla u|^2.
\end{eqnarray} From now on, we adapt some ideas outlined by Cheng and Zhou in \cite{cheng2017eigenvalues} to prove the lower bound estimate. Fix a point $p\in M$ and let $B_p (r)$ be the geodesic ball centered at $p$ with radius $r.$ Moreover, choose a cut-off function $\phi$ such that $\phi \equiv 1$ in $B_r$, $\phi \equiv 0$ outside $B_{r+1}$ and $|\nabla\phi| \leq 1.$ Besides, multiplying (\ref{eq2}) by $\phi^2$ and integrating over $M,$ one obtains that
    \begin{eqnarray}\label{eq3}
        \frac{1}{2}\int_M \phi^2\Delta_f |\nabla u|^2d\mu &=& \int_M \phi^2|\nabla^2 u|^2d\mu + (\lambda-\delta)\int_M \phi^2|\nabla u|^2d\mu \nonumber\\&&+ \rho\int_M \phi^2R|\nabla u|^2d\mu ,
    \end{eqnarray} where $d\mu = e^{-f}dV.$

   Next, by the weighted Green formula, one deduces that
    \begin{align}\label{eq4}
                \int_M \phi^2\Delta_f |\nabla u|^2d\mu &= -\int_M\langle\nabla\phi^2,\nabla|\nabla u|^2\rangle d\mu\nonumber\\
        &=-4\int_M\phi\langle\nabla_{\nabla\phi}\nabla u,\nabla u\rangle d\mu\\
        &=- 4\int_M \phi\nabla^2 u(\nabla\phi,\nabla u)d\mu.\nonumber
    \end{align} At the same time, from Young's inequality, we get

\begin{align}\label{eq5}
              -2\phi(\nabla^2 u)(\nabla\phi,\nabla u) &= -2\sum_{i,j=1}^n \phi(\nabla^2 u)_{ij}\nabla_i\phi\nabla_ju\nonumber\\
              &\leq \sum_{i,j=1}^n\left(\varepsilon\phi^2(\nabla^2 u)_{ij}^2 + \frac{1}{\varepsilon}(\nabla_i\phi)^2(\nabla_j u)^2\right)\nonumber\\
              &=\varepsilon\phi^2|\nabla^2 u|^2 + \frac{1}{\varepsilon}|\nabla\phi|^2|\nabla u|^2.
\end{align} Plugging (\ref{eq5}) into (\ref{eq4}), we obtain

\begin{equation}
    \int_M\phi^2\Delta_f|\nabla u|^2d\mu \leq 2\varepsilon\int_M\phi^2|\nabla^2 u|^2d\mu + \frac{2}{\varepsilon}\int_M|\nabla\phi|^2|\nabla u|^2d\mu.
\end{equation} This jointly with (\ref{eq3}) gives
\begin{eqnarray*}
    (1-\varepsilon)\int_{B_{r+1}}\phi^2|\nabla^2 u|^2d\mu &\leq & \frac{1}{\varepsilon}\int_{B_{r+1}}|\nabla\phi|^2|\nabla u|^2d\mu + (\delta-\lambda)\int_{B_{r+1}}\phi^2|\nabla u|^2d\mu\nonumber\\&&- \rho\int_{B_{r+1}}\phi^2R|\nabla u|^2d\mu\\
    &\leq& \frac{1}{\varepsilon}\int_{B_{r+1}}|\nabla\phi|^2|\nabla u|^2d\mu + (\delta-\lambda)\int_{B_{r+1}}\phi^2|\nabla u|^2d\mu.
\end{eqnarray*} Since the right hand side of the above expression is integrable, letting $r\rightarrow \infty,$ we obtain that $\int_M|\nabla^2 u|^2d\mu < \infty.$

Also, by Young's inequality, one concludes that

\begin{align}\label{eq6}
    | 2\phi(\nabla^2 u)(\nabla\phi,\nabla u)| &= \bigg|2\sum_{i,j=1}^n \phi(\nabla^2 u)_{ij}\nabla_i\phi\nabla_ju\bigg|\nonumber\\
    &\leq \sum_{i,j=1}^n\varepsilon\phi^2(\nabla^2 u)_{ij}^2|\nabla_i\phi| + \frac{1}{\varepsilon}|\nabla_i\phi|(\nabla_j u)^2\nonumber\\
    &\leq \varepsilon\phi^2|\nabla^2 u|^2|\nabla\phi| + \frac{1}{\varepsilon}|\nabla\phi||\nabla u|^2\nonumber,
\end{align} which combined with (\ref{eq4}) yields
\begin{eqnarray}
\label{plkjl1}
    \bigg|\int_M\phi^2\Delta_f|\nabla u|^2d\mu\bigg| &\leq & 2\varepsilon\int_M\phi^2|\nabla^2 u|^2|\nabla\phi|d\mu + \frac{2}{\varepsilon}\int_M|\nabla\phi||\nabla u|^2d\mu\nonumber\\
    &\leq & 2\varepsilon\int_{B_{r+1}\setminus B_r}|\nabla^2 u|^2d\mu + \frac{2}{\varepsilon}\int_{B_{r+1}\setminus B_r}|\nabla u|^2d\mu.
\end{eqnarray} Letting $r\rightarrow \infty,$ and noting that the right hand side of (\ref{plkjl1}) is integrable, we conclude that it converges to zero. Thus 
\begin{equation*}
\int_M \Delta_f|\nabla u|^2d\mu = 0.
\end{equation*} Therefore, by using (\ref{eq3}), we arrive at
\begin{equation}\label{eq7}
    0 = \int_M |\nabla^2 u|^2d\mu + (\lambda-\delta)\int_M |\nabla u|^2d\mu + \rho\int_M R|\nabla u|^2d\mu.
\end{equation} Taking into account that $u$ is not constant, we conclude that $\delta \geq \lambda,$ which proves the asserted lower bound.

Now, we need to analyze the equality case, i.e., $\delta = \lambda.$ To do so, we claim that if equality holds, then $R\equiv 0.$ We argue by contradiction by supposing that $\delta = \lambda$ and $R\not\equiv 0.$ In this situation, since $R \geq 0,$ there exists a point $x \in M$ such that $R(x) > 0.$ Hence, from (\ref{eq7}), we have $\nabla^2u \equiv 0,$ which implies that  $\nabla u$ is a nontrivial parallel vector field. Therefore, $M^n$ must be isometric to a product manifold $M^n= \Sigma^{n-1}\times \mathbb{R},$ for some complete manifold $\Sigma^{n-1}.$ In particular, the function $u$ is constant on the level set $\Sigma\times \{t\},$ with $t\in \Bbb{R},$ and $|\nabla u|$ only depends on $\Bbb{R}.$ Moreover, the scalar curvature of $M^n$ only depends on the scalar curvature of $\Sigma^{n-1}.$ In another words, we have 

\begin{equation*}
    R|\nabla u|^2(p,q) = R_{\Sigma}(p)|\nabla u|^2(q),
\end{equation*} where $(p,\,q) \in \Sigma^{n-1}\times\mathbb{R},$ $R$ stands for the scalar curvature of $M^n$ and $R_{\Sigma}$ is the scalar curvature of $\Sigma^{n-1}.$ Besides, taking into account that $R(x) = R_{\Sigma}(p),$ where $x=(p,q)\in  \Sigma^{n-1}\times \mathbb{R},$ one sees that $R_{\Sigma}(p) > 0.$ At the same time, since $\rho R\geq 0$ and $\delta = \lambda,$ we have from (\ref{eq7}) that

\begin{equation}
    0 = R|\nabla u|^2(p,s) = R_{\Sigma}(p)|\nabla u|^2(s),\quad \forall s \in \mathbb{R}
\end{equation} Therefore,  $|\nabla u|^2 (s) =0, \,\, \forall s\in\mathbb{R},$ which contradicts the fact that $\nabla u$ is a nontrivial vector field and then, such a point $x$ does not exist. Consequently, $R\equiv 0.$

To conclude, since $R\equiv 0,$ we obtain that $(M^n,\,g,\,f,\,\lambda)$ must satisfy $Ric+\nabla^2 f = \lambda g$ then it is a gradient shrinking Ricci soliton with vanishing scalar curvature, from this we conclude that $(M^n,\,g,\,f,\lambda)$ satisfies $Ric+\nabla^2 f = \lambda g$ with zero scalar curvature. Hence, $M^n$ is the Gaussian shrinking soliton (see \cite{chen,pigola}), up to scaling. This concludes the proof of the theorem.
\end{proof}

\subsection{Proof of Theorem  \ref{theo4}}
\label{sec4}
\begin{proof}
Here, we adapt some ideas used by Cheng, Ribeiro and Zhou in \cite{cheng2022volume}, see also \cite{li2012geometric}. To begin with, multiplying \eqref{eqvar1} by $r^2$ and integrating from $\varepsilon$ to $r,$ we get
\begin{equation*}
\int_{\varepsilon}^{r}t^2w' dt + \frac{1}{n-1}\int_{\varepsilon}^{r}t^2w^2dt + \int_{\varepsilon}^{r} t^2Ric\Big(\frac{\partial}{\partial t},\frac{\partial}{\partial t}\Big) dt \le 0.
\end{equation*} This implies
\begin{equation*}
\int_{\varepsilon}^r(t^2w)'dt \le \int_{\varepsilon}^r\Big(2tw-\frac{1}{n-1}t^2w^2\Big)dt + -\int_{\varepsilon}^r t^2Ric\Big(\frac{\partial}{\partial t},\frac{\partial}{\partial t}\Big)dt.
\end{equation*}
Now, letting $\varepsilon \rightarrow 0,$ it follows that
\begin{align*}
    r^2w &\le \int_{0}^r\Big(2tw-\frac{1}{n-1}t^2w^2\Big)dt + -\int_{0}^r t^2Ric\Big(\frac{\partial}{\partial t},\frac{\partial}{\partial t}\Big)dt\\
    &= \int_{0}^{r}\Big(-\frac{1}{n-1}(tw-(n-1))^2 + (n-1)\Big)dt  -\int_{0}^r t^2Ric\Big(\frac{\partial}{\partial t},\frac{\partial}{\partial t}\Big)dt\\
    &\leq \int_{0}^r (n-1) - t^2Ric\Big(\frac{\partial}{\partial t},\frac{\partial}{\partial t}\Big)dt.
\end{align*}
Thus, one deduces that
\begin{equation}\label{eq8}
    \left(\log\Big(\frac{J(\theta,r)}{r^{n-1}}\Big)\right)' \le -\frac{1}{r^2}\int_{0}^r t^2Ric\Big(\frac{\partial}{\partial t},\frac{\partial}{\partial t}\Big)dt.
\end{equation}
Again, by integrating (\ref{eq8}) from $\varepsilon$ to $r,$ we arrive at
\begin{equation*}
\int_{\varepsilon}^r\left(\log\frac{J(\theta,r)}{s^{n-1}}\right)'ds \le -\int_{\varepsilon}^r\left[\frac{1}{s^2}\int_{0}^s t^2Ric\Big(\frac{\partial}{\partial t},\frac{\partial}{\partial t}\Big)dt\right]ds.
\end{equation*} Since $\displaystyle\lim_{r\to0}\frac{J(\theta,r)}{r}=1,$ by letting $\varepsilon \rightarrow 0$ and integrating by parts, we obtain
\begin{equation}\label{eq9}
    \log\Big(\frac{J(\theta,r)}{r^{n-1}}\Big) \le \frac{1}{r}\int_{0}^rt^2Ric\Big(\frac{\partial}{\partial t},\frac{\partial}{\partial t}\Big)dt - \int_{0}^r tRic\Big(\frac{\partial}{\partial t},\frac{\partial}{\partial t}\Big)dt.
\end{equation} Hence, combining (\ref{eq8}) and (\ref{eq9}), we get
\begin{equation}\label{eq10}
    \left(r\log\frac{J(\theta,r)}{r^{n-1}}\right)' \le -\int_{0}^r
tRic\Big(\frac{\partial}{\partial t},\frac{\partial}{\partial t}\Big)dt.\end{equation}

On the other hand, since $M^n$ is a gradient $\rho$-Einstein soliton, we have 
\begin{equation*}
Ric\Big(\frac{\partial}{\partial t},\frac{\partial}{\partial t}\Big) = \rho R + \lambda - f''(t),
\end{equation*} 
where $f(t) = f(\gamma(t))$, $0\le t \le r$, and $\gamma$ is a minimizing geodesic with $\gamma(0) = p$. Thus, (\ref{eq10}) becomes
\begin{align*} 
   \left(r\log\frac{J(\theta,r)}{r^{n-1}}\right)' &\le  -\int_{0}^r t(\rho R + \lambda - f''(t))dt \\
   &\leq -\frac{\lambda r^2}{2} + rf'(r) - f(r) + f(0) -\int_{0}^r t\rho Rdt.
\end{align*} From this, it follows that
\begin{align}\label{eq12}
\left(r\log\frac{J(\theta,r)}{r^{n-1}}\right) &\le -\frac{\lambda r^3}{6} +  \int_{0}^r tf'(t)dt - \int_{0}^r f(t) dt + rf(0) -\int_{0}^r\int_{0}^s t\rho R dtds\nonumber\\
&= -\frac{\lambda r^3}{6} +  \int_{0}^r (tf)'(t)dt - 2\int_{0}^r f(t) dt + rf(0) -\int_{0}^r\int_{0}^s t\rho R dtds\\
&= -\frac{\lambda r^3}{6} + rf(r) - 2\int_{0}^r f(t) dt + rf(0) -\int_{0}^r\int_{0}^s t\rho R dtds.\nonumber
\end{align}
Thus, by using the hypothesis, one deduces that
\begin{equation}\label{eq13}
\left(r\log\frac{J(\theta,r)}{r^{n-1}}\right) \le -\frac{\lambda r^3}{6} +  rf(r) + rf(0) - 2\int_{0}^r f(t)dt,
\end{equation} which yields 
\begin{equation*}
   J(\theta,r) \le e^{(-\frac{\lambda r^2}{6} + f(r) + f(p) - \frac{2}{r}\int_{0}^rf(t)dt)} r^{n-1}. 
\end{equation*} Consequently,
\begin{eqnarray}
\label{eq14}
\Vol(B_{p}(r)) &=&  \int_{\mathbb{S}^{n-1}}\int_{0}^{\min\{r,\rho(\theta)\}}J(\theta, r)dr d\theta \nonumber\\&\le&  \int_{\mathbb{S}^{n-1}}\int_{0}^{\min\{r,\rho(\theta)\}}e^{\Phi}r^{n-1}dr d\theta\nonumber\\
&\le &  \int_{\mathbb{S}^{n-1}}\int_{0}^{r}e^{\Phi}r^{n-1}dr d\theta,
\end{eqnarray} where  $\Phi=-\frac{\lambda r^2}{6} + f(\theta, r)+f(p)-\frac{2}{r}\int_{0}^{r}f(\theta, s)ds$ and $\rho(\theta)$ stands for the cut-locus radius in the direction of $\theta.$ Finally, if equality holds in (\ref{eq14}), then it must also hold in (\ref{eq13}). Therefore,
\begin{equation*}
    -\int_{0}^r\int_{0}^s t\rho Rdtds = 0,
\end{equation*}
 for all $0<r<\rho(\theta)$. Taking into account that $\rho R \geq 0,$ for each point $p\in M$ there exists an open neighborhood $U_p$ such that  $R\equiv 0$ in $U_p$ and hence, since $p$ is arbitrary, one deduces that $R \equiv 0$ in $M.$ Similar to the proof of the previous theorem, we conclude that $(M^n,g,f,\lambda)$ must satisfy $Ric+\nabla^2 f = \lambda g$ with zero scalar curvature. Consequently, $M^n$ is the Gaussian shrinking soliton (see \cite{chen,pigola}). So, the proof is completed. 

\end{proof}

\subsection{Proof of Theorem \ref{theo5} and Corollary \ref{cor2}}
\label{sec5}
In this subsection, we present the proof of Theorem \ref{theo5} and Corollary \ref{cor2}. For the sake of convenience, we restate the theorem here.

\begin{theorem}[Theorem \ref{theo5}]
    Let $(M^n,\,g,\, f)$ be a complete $n$-dimensional noncompact gradient $\rho$-Einstein soliton. Then we have:
    \begin{equation}
        Vol_{f}(B_p(r)) \le  \int_{\mathbb{S}^{n-1}} \int_0^{r}e^{\Psi}r^{n-1}drd\theta,
    \end{equation}
    where  $\Psi= -\frac{\lambda r^2}{6}+ f(p)-\frac{2}{r}\int_{0}^{r}f(\theta, s)ds -\frac{1}{r}\int_{0}^r\int_{0}^s t\rho Rdtds$. 
\end{theorem}

\begin{proof} Initially, by using the volume form in polar coordinates $dV_{exp_{p}(r\theta)}=J(\theta,r)drd\theta$ for $\theta\in S_{p}M$ and a fixed point $p,$ we may denote the weighted volume form as $dV_f=e^{-f(r,\theta)}J(r,\theta)drd\theta$. It follows from the proof of Theorem \ref{theo4} that 
\begin{equation}\label{eq20}
r\log\frac{J(\theta,r)}{r^{n-1}} \le -\frac{\lambda r^3}{6} + rf(r) - 2\int_{0}^r f(t) dt + rf(0) -\int_{0}^r\int_{0}^s t\rho Rdtds.
\end{equation}
Consequently,
    \begin{equation*}
   J(\theta,r) \le e^{(-\frac{\lambda r^2}{6} + f(r) + f(p) - \frac{2}{r}\int_{0}^rf(t)dt-\frac{1}{r}\int_{0}^r\int_{0}^s t\rho Rdtds)} r^{n-1}, 
\end{equation*}
or equivalently,
 \begin{equation}
   J(\theta,r)e^{-f(r,\theta)} \le e^{(-\frac{\lambda r^2}{6}  + f(p) - \frac{2}{r}\int_{0}^rf(t)dt-\frac{1}{r}\int_{0}^r\int_{0}^s t\rho Rdtds)} r^{n-1}. 
\end{equation} Hence, by integration, one sees that
\begin{align}\label{eq100}
\Vol_{f}(B_{p}(r)) &=  \int_{\mathbb{S}^{n-1}}\int_{0}^{\min\{r,\rho(\theta)\}}J(\theta, r)e^{-f(r,\theta)}dr d\theta \nonumber\\&\le  \int_{\mathbb{S}^{n-1}}\int_{0}^{\min\{r,\rho(\theta)\}}e^{\Psi}r^{n-1}dr d\theta\nonumber\\
&\le  \int_{\mathbb{S}^{n-1}}\int_{0}^{r}e^{\Psi}r^{n-1}dr d\theta,
\end{align}
where $\Psi= -\frac{\lambda r^2}{6}+ f(p)-\frac{2}{r}\int_{0}^{r}f(\theta, s)ds -\frac{1}{r}\int_{0}^r\int_{0}^s t\rho Rdtds$, which finishes the proof of the theorem. 
\end{proof}

\subsubsection{Proof of Corollary  \ref{cor2}}

\begin{proof} To prove the corollary, it suffices to estimate the function $\Psi.$ Indeed, since $\rho R \geq \delta,$ we have 
\begin{equation*}
    Ric_f \geq (\delta + \lambda)g.
\end{equation*} So, by \cite[Theorem 0.4]{munteanu2015topology}, we get the following estimate
\begin{equation*}
    f(x) \geq \frac{\delta+\lambda}{2}r(x)^2 - ar(x),
\end{equation*}
for some positive constant $a.$ This therefore implies that
\begin{equation}\label{eq15}
    -\frac{2}{r}\int_{0}^{r}f(r,\theta)drd\theta \leq -\frac{(\delta+\lambda)r^2}{3} + ar.
\end{equation}

On the other hand, our assumption implies
\begin{equation}\label{eq16}
    -\frac{1}{r}\int_{0}^r\int_{0}^s t\rho Rdtds \leq -\frac{\delta}{r}\int_{0}^r\int_{0}^{s}tdtds = -\frac{\delta r^2}{6}.
\end{equation}
Plugging (\ref{eq15}) and (\ref{eq16}) into the expression of $\Psi$ in Theorem \ref{theo5}, we conclude that there exists a constant $c > 0$ such that
\begin{equation}
    \Psi \leq -c r^2 + ar.
\end{equation} Finally, taking $r_0$ large enough, it follows that
\begin{equation*}
    Vol_{f}(B_p(r)) \leq \int_{\mathbb{S}^{n-1}}\int_{0}^{r}e^{\frac{r^2}{2}}r^{n-1}dr d\theta,
\end{equation*} as asserted. 
\end{proof}

	\noindent{\bf Conflict of Interest:} There is no conflict of interest to disclose.
	
	\

\noindent{\bf Data Availability:} Not applicable.
	
	\
	
\noindent{{\bf Acknowledgments.}} The author wishes to thank the referees for their careful reading and valuable suggestions.


\begin{thebibliography}{BB}
     \bibitem{agila2024geometric} E. Agila and J. Gomes: Geometric and analytic results for Einstein solitons. {\it Math. Nachr.} 297 (2024) 2855--2872.
    \bibitem{bakry2006diffusions} D. Bakry and M. \'Emery: Diffusions Hypercontractives. {\it Lect. Notes Math}. 1123 (1985) 177--206.
    \bibitem{barros2015bounds} A. Barros, R. Batista and E. Ribeiro Jr: Bounds on volume growth of geodesic balls for Einstein warped products. {\it Proc. Amer.  Math. Soc.} 143 (2015) 4415--4422.
   \bibitem{BRR} R. Batista, M. Ranieri and E. Ribeiro Jr.: Remarks on complete noncompact Einstein warped products. {\it Comm. Anal. Geom}. v. 28, n. 3 (2020) 547--563.
    \bibitem{bishop} R. Bishop and R Crittenden: Geometry of manifolds. New York: Academic Press, (1964).    
    \bibitem{borges2024rigidity} V. Borges: Rigidity of Bach-flat gradient Schouten solitons. {\it Manuscr. Math}. 175 (2024) 1--11.
    \bibitem{borges2022complete} V. Borges: On complete gradient Schouten solitons. {\it Nonlinear Anal.} 221 (2022) 112883.
    \bibitem{bourguignon1981ricci} J.-P. Bourguignon: Ricci curvature and Einstein metrics. Global Differential Geometry and Global Analysis: Proceedings of the Colloquium Held at the Technical University of Berlin. (1979) 42--63. 
    \bibitem{calabi} E. Calabi: On manifolds with non-negative Ricci curvature II. {\it Notices Amer. Math. Soc.} 22 (1975) A205.
    \bibitem{cao2010complete} H.-D. Cao and D. Zhou: On complete gradient shrinking Ricci solitons. {\it J. Differ. Geom.} 85 (2010) 175-186.
    \bibitem{catino2015rigidity} G. Catino, L. Mazzieri and S. Mongodi: Rigidity of gradient Einstein shrinkers. {\it Commun. Contemp.  Math.} 17 (2015) 155046.
    \bibitem{catino2016gradient} G. Catino and L. Mazzieri: Gradient Einstein solitons. {\it Nonlinear Anal.} 132 (2016) 66--94.
    \bibitem{catino2017ricci} G. Catino, Z. Djadli, C. Mantegazza and L. Mazzieri: The Ricci--Bourguignon flow. {\it Pac. J. Math.} 287 (2017) 337--370
    \bibitem{chan2022volume} P.-Y. Chan, Z. Ma and Y. Zhang: Volume growth estimates of gradient Ricci solitons. {\it J. Geom. Anal.} 32 (2022) 291. 
    \bibitem{chen} B.-L. Chen: Strong uniqueness of the Ricci flow. {\it J. Differ. Geom.} 82 (2009) 363--382.    
    \bibitem{cheng2017eigenvalues} X. Cheng and D. Zhou: Eigenvalues of the drifted Laplacian on complete metric measure spaces. {\it Commun. Contemp. Math.} 19 (2017), 165001.
    \bibitem{cheng2018spectral} X. Cheng and D. Zhou: Spectral properties and rigidity for self-expanding solutions of the mean curvature flows. {\it Math. Ann.} 371 (2018) 371--389.
    \bibitem{cheng2022volume} X. Cheng, E. Ribeiro Jr and D. Zhou: Volume growth estimates for Ricci solitons and quasi-Einstein manifolds. {\it J. Geom. Anal.} 32 (2022), 62.
    \bibitem{conrado2024wasserstein} F. Conrado and D. Zhou: The Wasserstein distance for Ricci shrinkers. {\it Intern. Math. Res. Not.} v. 2024 (2024), 10485–10502.
    \bibitem{cunha2023ricci} A. Cunha, R. Lemos and F. Roing: On Ricci-Bourguignon solitons: Triviality, uniqueness and scalar curvature estimates. {\it J. Math. Anal. Appl.} 526 (2023), 127333.
    \bibitem{feitosa2017construction} F. Feitosa, A. Filho and J. Gomes: On the construction of gradient Ricci soliton warped product. {\it Nonlinear Anal.} 161 (2017) 30--43.
    \bibitem{feitosa2019gradient} F. Feitosa, A. Filho, J. Gomes and R. Pina: Gradient Ricci almost soliton warped product. {\it J. Geom. Phys.} 143 (2019), 22-32.
    \bibitem{gomes2024gradient} J. Gomes and W. Tokura: Gradient Einstein-type warped products: rigidity, existence and nonexistence results via a nonlinear PDE. {\it Nonlinear Anal.} v. 255 (2025) 113759.
    \bibitem{Hamilton} R. Hamilton: Three-manifolds with positive Ricci curvature. {\it J. Differ. Geom.} 17 (1982), no. 2, 255--306.
    \bibitem{hein2014new} H.-J. Hein and A. Naber: New Logarithmic Sobolev inequalities and an $\varepsilon$-regularity theorem for the Ricci flow. {\it Commun. Pure Appl. Math.} 67 (2014) 1543--1561.
    \bibitem{hamilton1993formations} R. Hamilton: The formations of singularities in the Ricci Flow. Surveys in differential geometry. 2(1993) 7--136.
    \bibitem{li2012geometric} P. Li: Geometric analysis. Cambridge Studies in Advanced Mathematics/Cambridge University Press. 134 (2012).
    \bibitem{lichnerowicz1958geometrie} A. Lichnerowicz: G{\'e}om{\'e}trie des groupes de transformations. (1958).
    \bibitem{morgan2005manifolds} F. Morgan: Manifolds with density. {\it Not. Amer. Math. Soc.} 52 (2005) 853--858.
    \bibitem{munteanu2009volume} O. Munteanu: The volume growth of complete gradient shrinking Ricci solitons. ArXiv:0904.0798v2 [math.DG]. (2009)
    \bibitem{Natasa} O. Munteanu and N. Sesum: On gradient Ricci solitons. {\it J. Geom. Anal.} 23 (2013) 539--561.
    \bibitem{munteanu2014geometry} O. Munteanu and J. Wang: Geometry of manifolds with densities. {\it Adv. Math.} 259(2014) 269--305.
    \bibitem{munteanu2015topology} O. Munteanu and J. Wang: Topology of K{\"a}hler Ricci solitons. {\it J. Differ. Geom.} 100 (2015), 109--128.
    \bibitem{obata1962certain} M. Obata: Certain conditions for a Riemannian manifold to be isometric with a sphere. {\it J. Math. Soc. Japan.} 14 (1962) 333--340.
    \bibitem{pigola} S. Pigola, M.  Rimoldi, A. Setti: Remarks on non-compact gradient Ricci solitons. {\it Math. Z.} 268 (2011) 777--790.
    \bibitem{shaikh2021some} A. Shaikh, A. Cunha and P. Mandal: Some characterizations of $\rho$-Einstein solitons. {\it J. Geom. Phys}. 166 (2021) 104270.
    \bibitem{wei2009comparison} G. Wei and W. Wylie: Comparison geometry for the Bakry-\'Emery Ricci tensor. {\it J. Differ. Geom.} 83 (2009) 337--405.
    \bibitem{yau1} S.-T. Yau: Some function-theoretic properties of complete Riemannian manifold and their applications to geometry. {\it Indiana Univ. Math. J.} 25 (1976) 659--670




\end{thebibliography}
\end{document}